\documentclass[11pt]{article}
\usepackage{amsmath,amsfonts,amsthm,amssymb, mathtools}
\usepackage{mathrsfs, graphicx,color,latexsym, tikz, calc,
}
\usepackage{tikz-cd}
\usepackage{graphicx}
\usepackage{epstopdf}
\usepackage{enumerate}
\usepackage{caption}
\usepackage{stmaryrd}
\usepackage{hyperref}

\usetikzlibrary{shadows}
\usetikzlibrary{patterns,arrows,decorations.pathreplacing}
\usepackage{ulem}

\newtheorem{theorem}{Theorem}[section]
\newtheorem{lemma}[theorem]{Lemma}

\newtheorem{conjecture}[theorem]{Conjecture}
\newtheorem{corollary}[theorem]{Corollary}

\newtheorem{claim}{Claim}[section]

\allowdisplaybreaks[4]

\title{\bf \Large }
\date{ }
\title{\bf \Large  
Proof of Frankl's conjecture on cross-intersecting families\footnote{This paper was published on Journal of Combinatorial Theory, Series A 216 (2025) 106062. 
E-mail addresses: \url{wuyjmath@163.com} (Y. Wu), \url{fenglh@163.com} (L. Feng), \url{ytli0921@hnu.edu.cn} (Y. Li)}}

\author{
{\small  Yongjiang Wu,\ \  Lihua Feng, \ \ Yongtao Li\footnote{Corresponding author}}\\[2mm]
\small School of Mathematics and Statistics, HNP-LAMA, Central South University\\
 \small Changsha, Hunan, 410083, China\\ }

\begin{document}
\maketitle
\begin{abstract}
Two families $\mathcal{F}$ and $\mathcal{G}$ are called 
cross-intersecting if for every $F\in \mathcal{F}$ and $G\in \mathcal{G}$, the intersection $F\cap G$ is non-empty. For any positive integers $n$ and $k$, let $\binom{[n]}{k}$ denote the family of all $k$-element subsets of $\{1,2,\ldots,n\}$.
Let $t, s, k, n$ be non-negative integers with $k \geq s+1$ and $n \geq 2 k+t$. In 2016, Frankl proved that if $\mathcal{F} \subseteq\binom{[n]}{k+t}$ and $\mathcal{G} \subseteq\binom{[n]}{k}$ are cross-intersecting families, and $\mathcal{F}$ is $(t+1)$-intersecting and $|\mathcal{F}| 
\geq 1$, then $|\mathcal{F}|+|\mathcal{G}| \leq\binom{n}{k}-\binom{n-k-t}{k}+1$. Furthermore, Frankl conjectured that under an additional condition $\binom{[k+t+s]} {k+t}\subseteq\mathcal{F}$, the following inequality holds:       
$$
|\mathcal{F}|+|\mathcal{G}| \leq\binom{k+t+s}{k+t}+\binom{n}{k}-\sum_{i=0}^s\binom{k+t+s}{i}\binom{n-k-t-s}{k-i}.
$$
In this paper, we prove this conjecture. 
The key ingredient is to establish a theorem for 
cross-intersecting families with a restricted universe. Moreover, we derive an analogous result for this conjecture. 
\end{abstract}



\section{Introduction}
For any two integers $a\leq b$, let  $[a, b]=\{a, a+1, \ldots, b\}$ and simply let $[n]=[1,n]$. Let $2^{[n]}$ denote the power set of $[n]$. For any $0 \leq k \leq n$, let $\binom{[n]}{k}$ denote the collection of all its $k$-element subsets. A family $\mathcal{F} \subseteq 2^{[n]}$ is called $k$-\textit{uniform} if $\mathcal{F} \subseteq \binom{[n]}{k}$. A family $\mathcal{F} \subseteq 2^{[n]}$ is called $t$-\textit{intersecting} if $|F\cap F^{\prime}|\geq t$
 for all $F, F^{\prime}\in \mathcal{F}$. If $t=1$, $\mathcal{F}$ is simply called \textit{intersecting}.

Erd\H{o}s, Ko and Rado \cite{E61} determined the maximum size of $k$-uniform intersecting families by proving that if $k\ge 2$, $n\ge 2k$ and 
 $\mathcal{F} \subseteq\binom{[n]}{k}$ is an intersecting family, then 
$$|\mathcal{F}|\leq\binom{n-1}{k-1}. $$
For $n > 2 k$, the equality holds if and only if $\mathcal{F}=\big\{F\in \binom{[n]}{k}: x\in F\big\}$ for some $x \in [n]$. 
Such a family is called a \textit{full star}. We shall mention that the case for $k=0$ or $k=1$ is trivial.
The Erd\H{o}s--Ko--Rado theorem is widely regarded as the cornerstone of extremal set theory and has many interesting applications and generalizations; see \cite{F87,H67,HZ2017} for related results and \cite{FT16,Ellis2021} for comprehensive surveys. A well-known generalization was given by Hilton and Milner \cite{H67} in 1967,  who showed that if $k \geq 2$, $n \geq 2 k$ and $\mathcal{F} \subseteq\binom{[n]}{k}$ is an intersecting family that is not a subfamily of a full star (i.e., non-trivial), then
$$
|\mathcal{F}|\leq\binom{n-1}{k-1}-\binom{n-k-1}{k-1}+1. 
$$
This improves greatly the bound in Erd\H{o}s--Ko--Rado's theorem. Moreover, Hilton and Milner
also provided a characterization of the non-trivial intersecting families with maximum size.
There are various intersection theorems that  strengthen Hilton--Milner's theorem in the literature; see \cite{H17,H24,KM17,K18,KZ2018}.

In 2013, 
Li, Chen, Huang, and Lih \cite{L13} proposed the following problem: let $n, m, k$ be positive integers with $n \geq2k>m\geq k$. Let $h(n,m,k)$ denote $\text{max}\{|\mathcal{F}| \}$, where $ \mathcal{F}\subseteq \binom{[n]}{k}$ runs over all intersecting families with $\binom{[m]}{k}\subseteq\mathcal{F}$. In other words, $h(n,m,k)$ is the maximum size of 
an intersecting family of $k$-sets of $[n]$ with clique number at least $m$. 
In the case $m=k$ and $k+1$, the Erd\H{o}s--Ko--Rado theorem and the Hilton--Milner theorem imply 
\[ h(n,k,k)= \binom{n-1 }{k-1} \] 
and 
\[ h(n,k+1,k)= \binom{n-1 }{ k-1} - \binom{n-k-1 }{k-1} +1. \]
Moreover, Li, Chen, Huang, and Lih \cite{L13} determined $h(n,m,k)$ for sufficiently large $n$ depending on $k$ or for $m=2k-1, 2k-2, 2k-3$. In 2020, Frankl \cite{F20} completely solved this problem by showing 
\[ h(n,m,k) = {m \choose k} + \sum_{i=m-k+1}^{k-1} {m-1 \choose i-1} {n-m \choose k-i}.  \] 
Two families  $\mathcal{F},\mathcal{G}\subseteq 2^{[n]}$ are  called \textit{cross-intersecting} if $|F\cap G|\geq 1$
 for any $F\in \mathcal{F}$ and $G\in \mathcal{G}$. 
 Cross-intersecting families are a variation of intersecting families. 
 In 1967, Hilton and Milner \cite{H67} proved that if $n \geq 2 k$ and  $\mathcal{F},\mathcal{G} \subseteq\binom{[n]}{k}$ are non-empty cross-intersecting families,  then
\begin{align}\label{f11}
    |\mathcal{F}|+|\mathcal{G}| \leq\binom{n}{k}-\binom{n-k}{k}+1. 
\end{align}
This result initiated the study of finding the maximum of the sum of sizes of cross-intersecting families.  
In 1992, Frankl and Tokushige \cite{FT92} established the following extension: let $ 2\le \ell \le k $ and $n\geq k+\ell$. If $\mathcal{F} \subseteq\binom{[n]}{k}$ and $\mathcal{G} \subseteq\binom{[n]}{\ell}$ are non-empty cross-intersecting families, then 
\begin{align}\label{f12}
     |\mathcal{F}|+|\mathcal{G}| \leq\binom{n}{k}-\binom{n-\ell}{k}+1. 
\end{align}
This bound can be attained by taking $\mathcal{G}$ as a single set. 
There are many other generalizations for cross-intersecting families in the past years; see, e.g.,  \cite{CKLT2023,CLLW2022,Fra2024,FT92,F98,FW2023,F24,Huang2019, SFQ2022,WZ2013}. 
In this paper, we are interested in the following generalization due to 
Frankl \cite{F16}.

\begin{theorem}[Frankl \cite{F16}] 
\label{F16}
 Let $t\ge 0, k \geq 1$ and $n \geq 2 k+t$ be  integers.  Let $\mathcal{F} \subseteq\binom{[n]}{k+t}$ and $\mathcal{G} \subseteq\binom{[n]}{k}$ be  cross-intersecting families.
\begin{itemize}
\item[\rm (i)] 
If $\mathcal{F}$ is $t$-intersecting, then
$$
|\mathcal{F}|+|\mathcal{G}| \leq\binom{n}{k}.
$$

\item[\rm (ii)] 
If $\mathcal{F}$ is $(t+1)$-intersecting and $|\mathcal{F}| \geq 1$, then
$$
|\mathcal{F}|+|\mathcal{G}| \leq\binom{n}{k}-\binom{n-k-t}{k}+1.
$$
\end{itemize}
\end{theorem}

The Hilton lemma \cite{H76}
states that for cross-intersecting families 
$\mathcal{F}$ and $\mathcal{G}$, 
we can always assume that $\mathcal{F}$ and $\mathcal{G}$ are the first initial segments in the lexicographic order; 
see \cite{FK2017} for a detailed proof. 
Therefore,  the case $t=0$ in Theorem \ref{F16} implies Hilton--Milner's result (\ref{f11}).
Moreover, the case $t=1$ in Theorem \ref{F16} can be applied to a theorem of Frankl \cite{F171} to establish a stability  result of the Katona theorem \cite{K64}, 
and it can also imply a recent result of 
Bulavka and Woodroofe \cite{BW2024}. 
Furthermore, the third author and Wu \cite{LW2024} sharpened Theorem \ref{F16} in the case $t=1$.
Subsequently, Wu \cite{W23} provided a generalization under the constraint $|\mathcal{F}|\geq 2$. 
 A more general extension was recently established by the current authors \cite{Wu-third} under $|\mathcal{F}|\ge r$ for every $r\le n-k-t +1$.
Let us mention that Theorem \ref{F16} and its generalizations are useful on the study of stability results for families with restricted intersection or union \cite{Wu-first,Wu-second}.

\section{Main results} 

In the same paper \cite{F16}, Frankl proposed the following conjecture.

\begin{conjecture}[Frankl \cite{F16}] 
\label{co1}
 Let $t,s\ge 0, k \geq s+1$ and $n \geq 2 k+t$ be integers.  Let $\mathcal{F} \subseteq\binom{[n]}{k+t}$ and $\mathcal{G} \subseteq\binom{[n]}{k}$ be  cross-intersecting families. If $\mathcal{F}$ is $(t+1)$-intersecting and $\binom{[k+t+s]}{k+t}\subseteq\mathcal{F}$, then
$$
|\mathcal{F}|+|\mathcal{G}| \leq\binom{k+t+s}{k+t}+\binom{n}{k}-\sum_{i=0}^s\binom{k+t+s}{i}\binom{n-k-t-s}{k-i}.
$$
\end{conjecture}

Note that setting $s=0$, we get back to (ii) in  Theorem \ref{F16}. The bound in Conjecture \ref{co1} is the best possible as demonstrated by 
the following families:
$$
\mathcal{F}_0=\binom{[k+t+s]}{k+t}~~\text{and}~~ \mathcal{G}_0=\left\{G\in \binom{[n]}{k}: |G\cap[k+t+s]|\geq s+1\right\}.
$$
Frankl \cite{F16} mentioned that he can prove Conjecture \ref{co1} for some small values of $k$ and also for $n> ck^2$, and he pointed out that proving it in the full range appears to be difficult.

In this paper, we give a complete proof of  Conjecture \ref{co1} by establishing a theorem for cross-intersecting families with a restricted universe. 
Our proof is motivated by the ideas of Frankl \cite{F20}. 
The second main theorem of this paper gives a variant of  Conjecture \ref{co1}.

\begin{theorem}\label{main6}
Let $s\geq 0$, $k \geq \ell\geq s+1$ and $n \geq k+\ell$ be integers.   Let $\mathcal{F} \subseteq\binom{[n]}{k}$ and $\mathcal{G} \subseteq\binom{[n]}{\ell}$ be  cross-intersecting families.
Suppose that $\mathcal{G}$ is intersecting and  $\binom{[\ell+s]}{\ell}\subseteq\mathcal{G}$. Then
$$
|\mathcal{F}|+|\mathcal{G}| \leq\binom{\ell+s}{\ell}+\binom{n}{k}-\sum_{i=0}^s\binom{\ell+s}{i}\binom{n-\ell-s}{k-i}.
$$
\end{theorem}

For Theorem \ref{main6}, setting $s=0$ we get the bound in (\ref{f12}) with additional assumption that $\mathcal{G}$ is intersecting. 
Here, we can drop the constraint that $\mathcal{F}$ is non-empty. 
In addition, setting $k=\ell$, we can get part (ii) in Theorem \ref{F16} for $t=0$. So we can view Theorem \ref{main6} as a  generalization of Hilton--Milner's result (\ref{f11}). 
As is well-known, for a family $\mathcal{G} \subseteq\binom{[n]}{\ell}$, if $\mathcal{G}$ is shifted (see the definition in next section or \cite{G16}) and non-trivial, then $[2,\ell+1]\in \mathcal{G}$ and $\binom{[\ell+1]}{\ell}\subseteq\mathcal{G}$. 
Therefore, Theorem \ref{main6} implies the following corollary.

\begin{corollary}\label{cm6}
Let $k \geq \ell\geq 1$ and $n \geq k+\ell$ be integers.   Let $\mathcal{F} \subseteq\binom{[n]}{k}$ and $\mathcal{G} \subseteq\binom{[n]}{\ell}$ be cross-intersecting families.
Suppose that $\mathcal{G}$ is shifted, intersecting and non-trivial. Then
$$
|\mathcal{F}|+|\mathcal{G}| \leq \ell+1+\binom{n}{k}-\sum_{i=0}^1\binom{\ell+1}{i}\binom{n-\ell-1}{k-i}.
$$
\end{corollary}

 Corollary \ref{cm6} could be viewed as a variant of \cite[Theorems 1.5 and 3.5]{Fra2024}  with additional assumption that $\mathcal{G}$ is intersecting and without the constraint that $\mathcal{F}$ is shifted and non-trivial.

\section{Proof of Conjecture \ref{co1}}   

\label{se3}

Let us introduce the shifting operation 
due to Erd\H{o}s, Ko and Rado \cite{E61}. 
Let  $\mathcal{F} \subseteq 2^{[n]}$ be a family and $1 \leq i<j \leq n$. The \textit{shifting operator} $s_{i, j}$ on $\mathcal{F}$ is defined as follows:
$$s_{i, j}(\mathcal{F})=\left\{s_{i, j}(F): F \in \mathcal{F}\right\},$$
where
$$
s_{i, j}(F)= \begin{cases}(F \backslash\{j\}) \cup\{i\} & \text { if } j \in F, i \notin F \text { and } (F \backslash\{j\}) \cup\{i\} \notin \mathcal{F}, \\ F & \text { otherwise. }\end{cases}
$$
Obviously, we have $\left|s_{i, j}(F)\right|=|F|$ and $\left|s_{i, j}(\mathcal{F})\right|=|\mathcal{F}|$.  A frequently used property  is that $s_{i, j}$ maintains the $t$-intersecting property of a family. 
A family $\mathcal{F} \subseteq 2^{[n]}$ is called \textit{shifted} if for all $ F \in \mathcal{F}$,  $i<j$ with $i\notin F$ and $j\in F$, then $(F \backslash\{j\}) \cup\{i\} \in \mathcal{F}$.  
It is well-known that every intersecting family can be transformed to a shifted intersecting family by applying shifting operations repeatedly.  
There are many nice properties on shifted families. For example, 
if $\mathcal{F}$ is a shifted family 
and $\{a_1,\ldots,a_k\}\in \mathcal{F}$ with $a_1<\cdots<a_k$, then for any set  $\{b_1,\ldots,b_k\}$ with $b_1<\cdots<b_k$ and $b_i\leq a_i$ for each $i\in[1,k]$, we have $\{b_1,\ldots,b_k\}\in \mathcal{F}$.

Let us recall the following well-known result.

\begin{lemma} \cite[Fact 2.1]{G23} \label{S3}
 Let $\mathcal{F} \subseteq 2^{[n]}$ and $\mathcal{G} \subseteq 2^{[n]}$ be  cross-intersecting, and $\mathcal{F}$ be $t$-intersecting.
Then $s_{i, j}(\mathcal{F})$ and $s_{i, j}(\mathcal{G})$ are also cross-intersecting, and $s_{i, j}(\mathcal{F})$ is $t$-intersecting.
\end{lemma}

Let  $\mathcal{F} \subseteq 2^{[n]}$ be a family. For each $i\in [n]$, we denote 
\begin{align*}
\mathcal{F}(\bar{i}) &=\left\{F \in \mathcal{F}: i \notin F \right\},\\ 
\mathcal{F}(i)&=\left\{F\backslash \{i\}: i \in F \in \mathcal{F}\right\}.
\end{align*}
The following lemma will help with the induction step.

\begin{lemma} \label{S4}
 Let $\ell\geq t\geq 0,k\geq 1$ and $n \geq 2 k+t$ be  integers.
Let $\mathcal{F} \subseteq\binom{[n]}{k+t}$ and $\mathcal{G} \subseteq\binom{[n]}{k}$ be shifted cross-intersecting families.  Then $\mathcal{F}(n)$ and $\mathcal{G}(n)$ are cross-intersecting. Moreover,  if $n>2k+t$ and $\mathcal{F}$ is $\ell$-intersecting, then $\mathcal{F}(n)$ is $\ell$-intersecting.
\end{lemma}

\begin{proof}
  For the first statement, let $F_1\in \mathcal{F}(n)$ and $G_1\in \mathcal{G}(n)$. Then $F_1\cup\{n\}\in \mathcal{F}$ and $G_1\cup\{n\}\in \mathcal{G}$.  Since $\mathcal{F}$ is shifted and
$
|F_1\cup G_1\cup\{n\}|\leq 2k+t-1\leq n-1,
$
 there exists $x\notin F_1\cup G_1\cup\{n\}$ such that $F_1\cup\{x\} \in \mathcal{F}$. Then 
$
|F_1\cap G_1|=|(F_1\cup\{x\})\cap (G_1\cup\{n\})|\geq 1$. Hence, $\mathcal{F}(n)$ and $\mathcal{G}(n)$ are cross-intersecting. 

For the second statement, let $E_1, E_2 \in \mathcal{F}(n)$.  Then $E_1\cup\{n\}, E_2\cup\{n\}\in \mathcal{F}$. Since $\mathcal{F}$ is shifted and
$
|E_1\cup E_2\cup\{n\}|\leq 2(k+t)-\ell< n,
$
there exists $x\notin E_1\cup E_2\cup\{n\}$ such that $E_1\cup\{x\} \in \mathcal{F}$. Recall that $\mathcal{F}$ is $\ell$-intersecting. It is immediate that
$
|E_1\cap E_2|=|(E_1\cup\{x\})\cap (E_2\cup\{n\})|\geq \ell,
$
as desired. 
\end{proof}

The following result is a key ingredient in our proof of Conjecture \ref{co1}. 

\begin{theorem}\label{main2}
Let $t, s$ be non-negative integers. Let $k \geq s+1$, $n \geq 2 k+t$ and $m> k+t+s$ be integers.  Let $\mathcal{F} \subseteq\left\{A\in \binom{[n]}{k+t}: |A\cap[m]|\geq t+s+1\right\}$ and $\mathcal{G} \subseteq\left\{B\in \binom{[n]}{k}: |B\cap[m]|\geq s+1\right\}$ be  cross-intersecting families. 
 Suppose that  $\mathcal{F}$ is $t$-intersecting. Then
$$
|\mathcal{F}|+|\mathcal{G}| \leq\binom{n}{k}-\sum_{i=0}^s\binom{m}{i}\binom{n-m}{k-i}.
$$
\end{theorem}

\begin{proof}
For notational convenience, 
we denote $\mathcal{R}(n, k, m):=\left\{B\in \binom{[n]}{k}: |B\cap[m]|\geq s+1\right\}$. It suffices to prove that
$|\mathcal{F}|+|\mathcal{G}| \leq|\mathcal{R}(n, k, m)|$.
If $m\geq n$, then Theorem \ref{F16} implies that 
$$
|\mathcal{F}|+|\mathcal{G}|\leq \binom{n}{k}=|\mathcal{R}(n, k, m)|.
$$
Next we assume that $m<n$.
Note that the shifting operation maintains the properties
$|F\cap[m]|\geq t+s+1$ and $|G\cap[m]|\geq s+1$. Combining with Lemma \ref{S3}, we may assume that $\mathcal{F}$ and $\mathcal{G}$ are shifted. 
For fixed $m$,  we are ready to apply double induction on $n$ and $k$.

For the case  $k = s+1$, we have that $\mathcal{F} \subseteq\binom{[n]}{s+t+1}$ and for any $F\in \mathcal{F}$, $|F\cap[m]|\geq t+s+1$. This implies that 
$\mathcal{F} \subseteq\binom{[m]}{s+t+1}$.
Similarly, we have $\mathcal{G} \subseteq\binom{[m]}{s+1}$. 
Since $m>2s+1+t$, note that $\mathcal{F}$ and $\mathcal{G}$ are  cross-intersecting, $\mathcal{F}$ is $t$-intersecting, by Theorem \ref{F16}, we get 
$$
|\mathcal{F}|+|\mathcal{G}| \leq \binom{m}{s+1}=|\mathcal{R}(n, s+1, m)|.
$$

For the case $n=2 k+t$, by Theorem \ref{F16}, we may assume that $m< 2k+t$. 
Since $2k+t-m<k-s $, it is enough to show that $
|\mathcal{F}|+|\mathcal{G}| \leq\binom{2k+t}{k}.
$
For any $F \in\binom{[2k+t]}{k+t}$, the cross-intersecting property of $\mathcal{F}$ and $\mathcal{G}$ implies that $F \notin \mathcal{F}$ or $[2 k+t] \backslash F \notin \mathcal{G}$. So 
$|\mathcal{F}|+|\mathcal{G}| \leq\binom{2k+t}{k}$, as desired.

Now assume that $k > s+1$ and $n > 2 k+t$.
Clearly,  $\mathcal{F}(\bar{n})\subseteq\left\{A\in \binom{[n-1]}{k+t}: |A\cap[m]|\geq t+s+1\right\}$ and $\mathcal{G}(\bar{n})\subseteq\mathcal{R}(n-1, k, m)$ are cross-intersecting. In addition,  $\mathcal{F}(\bar{n})$ is $t$-intersecting.
By induction hypothesis, we get
\begin{align}\label{f21}
|\mathcal{F}(\bar{n})|+|\mathcal{G}(\bar{n})| \leq\binom{n-1}{k}-\sum_{i=0}^s\binom{m}{i}\binom{n-1-m}{k-i}.
\end{align}
Since $n>m$ and $k > s+1$, we have 
$
\mathcal{F}(n)\subseteq\left\{A^*\in \binom{[n-1]}{k+t-1}: |A^*\cap[m]|\geq t+s+1\right\}
$
and
$
\mathcal{G}(n)\subseteq\mathcal{R}(n-1, k-1, m).
$
By Lemma \ref{S4}, $\mathcal{F}(n)$ and $\mathcal{G}(n)$ are cross-intersecting, and $\mathcal{F}(n)$ is $t$-intersecting.
Note that $m>k-1+t+s$. By induction hypothesis, we get
\begin{align}\label{f22}
|\mathcal{F}(n)|+|\mathcal{G}(n)| \leq|\mathcal{R}(n-1, k-1, m)|=\binom{n-1}{k-1}-\sum_{i=0}^s\binom{m}{i}\binom{n-1-m}{k-1-i}.
\end{align}
The inequalities (\ref{f21}) and (\ref{f22}) together imply Theorem \ref{main2}.
\end{proof}

Now we are in a position to prove Conjecture \ref{co1}. 

\medskip 

\noindent{\bf Proof of Conjecture \ref{co1}.}
First of all, we show the following claim. 

\begin{claim}\label{le1}
For any $F\in \mathcal{F}$ and $G\in \mathcal{G}$, we have $|F\cap[k+t+s]|\geq t+s+1$ and $|G\cap[k+t+s]|\geq s+1$.
\end{claim} 

\begin{proof}[Proof of claim]
Suppose that there exists $F_0\in \mathcal{F}$ such that $|F_0\cap[k+t+s]|\leq t+s$. Since $[k+t]\in\mathcal{F}$ and $\mathcal{F}$ is $(t+1)$-intersecting, we have $|F_0\cap[k+t]|\geq t+1$. Let $T\subseteq F_0\cap[k+t+s]$ with $|T|=t$ and $K\subseteq [k+t+s]\backslash F_0$ with $|K|=k$.
Since $\binom{[k+t+s]}{k+t}\subseteq\mathcal{F}$, we have $K\cup T \in \mathcal{F}$. But $|(K\cup T)\cap F_0|=t$, contradicting with that $\mathcal{F}$ is $(t+1)$-intersecting. 
In addition, suppose that there exists $G_0\in \mathcal{G}$ such that $|G_0\cap[k+t+s]|\leq s$. Since $\binom{[k+t+s]}{k+t}\subseteq\mathcal{F}$, we get that there is $F_1\in\mathcal{F}$ such that $F_1\cap G_0=\emptyset$,  contradicting with that $\mathcal{F}$ and $\mathcal{G}$ are cross-intersecting. 
\end{proof}

By Claim \ref{le1}, we have 
$$\mathcal{F} \subseteq\left\{A\in \binom{[n]}{k+t}: |A\cap[k+t+s]|\geq t+s+1\right\},$$ 
and 
$$ \mathcal{G} \subseteq\left\{B\in \binom{[n]}{k}: |B\cap[k+t+s]|\geq s+1\right\}.$$ 
The point is that the shifting operation maintains the properties  
$|F\cap[k+t+s]|\geq t+s+1$, $|G\cap[k+t+s]|\geq s+1$ and $\binom{[k+t+s]}{k+t}\subseteq\mathcal{F}$. In view of this fact and Lemma \ref{S3}, we may assume that $\mathcal{F}$ and $\mathcal{G}$ are shifted. 
Next we are ready to apply induction on $n$.

For the case $n=2 k+t$, note that $\binom{k+t+s}{k+t}+\binom{n}{k}-\sum_{i=0}^s\binom{k+t+s}{i}\binom{n-k-t-s}{k-i}=\binom{2k+t}{k},
$
so it suffices to show that $
|\mathcal{F}|+|\mathcal{G}| \leq\binom{2k+t}{k}.
$
For any $F \in\binom{[2k+t]}{k+t}$, the cross-intersecting property of $\mathcal{F}$ and $\mathcal{G}$ implies that $F \notin \mathcal{F}$ or $[2 k+t] \backslash F \notin \mathcal{G}$. It follows that
$|\mathcal{F}|+|\mathcal{G}| \leq\binom{2k+t}{k}$, as desired.

For the case $k = s+1$, we have
$\mathcal{F} \subseteq\binom{[2s+t+1]}{s+t+1}$ and $\mathcal{G} \subseteq\binom{[2s+t+1]}{s+1}$. 
Since $\binom{[2s+t+1]}{s+1+t}\subseteq\mathcal{F}$, we get $\mathcal{F} =\binom{[2s+t+1]}{s+t+1}$. Note that $ \mathcal{G} \subseteq\left\{B\in \binom{[2s+t+1]}{s+1}: |B\cap[2s+t+1]|\geq s+1\right\}$. Thus the result follows.

Next assume that $n > 2 k+t$ and $k>s+1$. Note that  $\mathcal{F}(\bar{n})$ is $(t+1)$-intersecting and $\binom{[k+t+s]}{k+t}\subseteq\mathcal{F}(\bar{n})$.
Moreover,  $\mathcal{F}(\bar{n})\subseteq\left\{A\in \binom{[n-1]}{k+t}: |A\cap[k+t+s]|\geq t+s+1\right\}$ and $\mathcal{G}(\bar{n})\subseteq\left\{B\in \binom{[n-1]}{k}: |B\cap[k+t+s]|\geq s+1\right\}$ are cross-intersecting. 
By induction hypothesis, we get
\begin{align}\label{f31}
|\mathcal{F}(\bar{n})|+|\mathcal{G}(\bar{n})| \leq\binom{k+t+s}{k+t}+\binom{n-1}{k}-\sum_{i=0}^s\binom{k+t+s}{i}\binom{n-k-t-s-1}{k-i}.
\end{align}
Since $n>k+t+s$ and $k > s+1$, we have 
$ \mathcal{F}(n)\subseteq\left\{A^*\in \binom{[n-1]}{k+t-1}: |A^*\cap[k+t+s]|\geq t+s+1\right\} $
and
$
\mathcal{G}(n)\subseteq\left\{B^*\in \binom{[n-1]}{k-1}: |B^*\cap[k+t+s]|\geq s+1\right\}.
$ 
By Lemma \ref{S4}, $\mathcal{F}(n)$ and $\mathcal{G}(n)$ are cross-intersecting, and $\mathcal{F}(n)$ is $(t+1)$-intersecting. Here, we just use the $t$-intersecting property of $\mathcal{F}(n)$.
Setting $m:=k+t+s>k-1+t+s$ in Theorem \ref{main2},  we get
\begin{align}\label{f32}
|\mathcal{F}(n)|+|\mathcal{G}(n)| \leq\binom{n-1}{k-1}-\sum_{i=0}^s\binom{k+t+s}{i}\binom{n-k-t-s-1}{k-1-i}.
\end{align}
Combining (\ref{f31}) with (\ref{f32}), 
we get the desired bound. 
$\hfill \square$\vspace{3mm}

\section{Proof of Theorem \ref{main6}} 

\label{se4}

To prove Theorem \ref{main6}, we need the following lemmas.

\begin{lemma} \cite[Proposition 1.3]{F87} 
\label{F87}
Let $k \geq \ell\geq 1$ and $n \geq k+\ell$ be integers.  Let $\mathcal{F} \subseteq\binom{[n]}{k}$ and $\mathcal{G} \subseteq\binom{[n]}{\ell}$ be  cross-intersecting families. Then
$$
|\mathcal{F}|+|\mathcal{G}| \leq\binom{n}{k}.
$$
\end{lemma}

\begin{lemma}\label{s51}
Let $s\geq 0$, $k \geq s+1$, $n > k+s$ and $ m> 2s+1$ be integers.  Let $\mathcal{F} \subseteq\left\{A\in \binom{[n]}{k}: |A\cap[m]|\geq s+1\right\}$ and $\mathcal{G} \subseteq\left\{B\in \binom{[n]}{s+1}: |B\cap[m]|\geq s+1\right\}$ be  cross-intersecting families.  Then
$$
|\mathcal{F}|+|\mathcal{G}| \leq\binom{n}{k}-\sum_{i=0}^s\binom{m}{i}\binom{n-m}{k-i}.
$$
\end{lemma}
\begin{proof}
If $m\geq n$, then Lemma \ref{F87} implies that 
$
|\mathcal{F}|+|\mathcal{G}|\leq \binom{n}{k},
$
as desired.
So we may assume that $m<n$.
We proceed the proof by applying induction on $n$.

For the case  $k = s+1$, we have $\mathcal{F} \subseteq\binom{[m]}{s+1}$. Note that $m>2s+1$. Since $\mathcal{F}\subseteq\binom{[m]}{s+1}$ and  $\mathcal{G} \subseteq\binom{[m]}{s+1}$ are cross-intersecting, it follows from Lemma \ref{F87} that
$
|\mathcal{F}|+|\mathcal{G}| \leq \binom{m}{s+1},
$
as desired.

For the case $n=k+s+1$, by Lemma \ref{F87}, we may assume that $m< k+s+1$. 
Since $k+s+1-m<k-s $, the upper bound in Lemma \ref{s51} is $
\binom{k+s+1}{k}$, which simply follows from the cross-intersecting property of $\mathcal{F}$ and $\mathcal{G}$.

Now assume that $k > s+1$ and $n > k+s+1$.
Clearly,  $\mathcal{F}(\bar{n})\subseteq\left\{A\in \binom{[n-1]}{k}: |A\cap[m]|\geq s+1\right\}$ and $\mathcal{G}(\bar{n})\subseteq\left\{B\in \binom{[n-1]}{s+1}: |B\cap[m]|\geq s+1\right\}$ are cross-intersecting. 
By induction hypothesis, we get
\begin{align}\label{f41}
|\mathcal{F}(\bar{n})|+|\mathcal{G}(\bar{n})| \leq\binom{n-1}{k}-\sum_{i=0}^s\binom{m}{i}\binom{n-1-m}{k-i}.
\end{align}
Since $n>m$ and $k > s+1$, we have 
$
\mathcal{F}(n)\subseteq\left\{A^*\in \binom{[n-1]}{k-1}: |A^*\cap[m]|\geq s+1\right\}
$
and
$
\mathcal{G}(n)=\emptyset.
$
It follows that
\begin{align}\label{f42}
|\mathcal{F}(n)|+|\mathcal{G}(n)| =|\mathcal{F}(n)|\leq\binom{n-1}{k-1}-\sum_{i=0}^s\binom{m}{i}\binom{n-1-m}{k-1-i}.
\end{align}
The inequalities (\ref{f41}) and (\ref{f42}) together imply Lemma \ref{s51}.
\end{proof}


To show Theorem \ref{main6}, we extend Lemma \ref{s51} to a more general setting. 

\begin{theorem}\label{main4}
Let $s\geq 0$, $k \geq \ell\geq s+1$, $n \geq k+\ell$ and $m> \ell+s$ be integers.  Let $\mathcal{F} \subseteq\left\{A\in \binom{[n]}{k}: |A\cap[m]|\geq s+1\right\}$ and $\mathcal{G} \subseteq\left\{B\in \binom{[n]}{\ell}: |B\cap[m]|\geq s+1\right\}$ be  cross-intersecting families. Then
$$
|\mathcal{F}|+|\mathcal{G}| \leq\binom{n}{k}-\sum_{i=0}^s\binom{m}{i}\binom{n-m}{k-i}.
$$
\end{theorem}

\begin{proof}
If $m\geq n$, then Lemma \ref{F87} implies that 
$
|\mathcal{F}|+|\mathcal{G}|\leq \binom{n}{k},
$
as required.
In what follows, we assume that $m<n$.
As mentioned early,  the shifting operation maintains the properties
$|F\cap[m]|\geq s+1$ and $|G\cap[m]|\geq s+1$. Combining with Lemma \ref{S3}, we may assume that $\mathcal{F}$ and $\mathcal{G}$ are shifted.
For fixed $m$,  we are ready to apply induction on $n$ and $k$ and $\ell$.

For the case $\ell = s+1$, it follows from Lemma \ref{s51}.

For the case $k=\ell$, it follows from 
Theorem \ref{main2} by setting $t=0$.

For the case $n=k+\ell$, we may assume that $m< k+\ell$. 
Since $k+\ell-m<k-s $, the bound in Theorem \ref{main4} is $
\binom{k+\ell}{k}$, which follows immediately,  since $\mathcal{F}$ and $\mathcal{G}$ are cross-intersecting.

Now assume that $k >\ell> s+1$ and $n > k+\ell$. Clearly, we have 
\[ \mathcal{F}(\bar{n})\subseteq\left\{A\in \binom{[n-1]}{k}: |A\cap[m]|\geq s+1\right\}, \]
 and 
\[ \mathcal{G}(\bar{n})\subseteq\left\{B\in \binom{[n-1]}{\ell}: |B\cap[m]|\geq s+1\right\}. \] 
Note that $\mathcal{F}(\bar{n})$ and 
$\mathcal{G}(\bar{n})$ are cross-intersecting. 
By induction hypothesis, we get 
\begin{align}\label{f43}
|\mathcal{F}(\bar{n})|+|\mathcal{G}(\bar{n})| \leq\binom{n-1}{k}-\sum_{i=0}^s\binom{m}{i}\binom{n-1-m}{k-i}.
\end{align}
Since $n>m$ and $k >\ell> s+1$, we have 
$
\mathcal{F}(n)\subseteq\left\{A^*\in \binom{[n-1]}{k-1}: |A^*\cap[m]|\geq s+1\right\}
$
and
$
\mathcal{G}(n)\subseteq\left\{B^*\in \binom{[n-1]}{\ell-1}: |B^*\cap[m]|\geq s+1\right\}.
$
By Lemma \ref{S4}, $\mathcal{F}(n)$ and $\mathcal{G}(n)$ are cross-intersecting.
Note that $m>\ell-1+s$. By induction hypothesis, we get
\begin{align}\label{f44}
|\mathcal{F}(n)|+|\mathcal{G}(n)| \leq\binom{n-1}{k-1}-\sum_{i=0}^s\binom{m}{i}\binom{n-1-m}{k-1-i}.
\end{align}
The inequalities (\ref{f43}) and (\ref{f44}) together imply Theorem \ref{main4}.
\end{proof}

With the help of  Theorem \ref{main4}, we can now prove Theorem \ref{main6}.\vspace{3mm}

\noindent{\bf Proof of Theorem \ref{main6}.} 
First of all, we claim that
$$\mathcal{F} \subseteq\left\{A\in \binom{[n]}{k}: |A\cap[\ell+s]|\geq s+1\right\},~\mathcal{G} \subseteq\left\{B\in \binom{[n]}{\ell}: |B\cap[\ell+s]|\geq s+1\right\}.$$
Suppose that there exists $F_0\in \mathcal{F}$ such that $|F_0\cap[\ell+s]|\leq s$. Since $\binom{[\ell+s]}{\ell}\subseteq\mathcal{G}$, we can select $G_0\in\mathcal{G}$ such that $F_0\cap G_0=\emptyset$,  contradicting with that $\mathcal{F}$ and $\mathcal{G}$ are cross-intersecting. In addition, suppose that there exists $G_1\in \mathcal{G}$ such that $|G_1\cap[\ell+s]|\leq s$. Since $\binom{[\ell+s]}{\ell}\subseteq\mathcal{G}$, we can select $G_2\in\mathcal{G}$ such that $G_1\cap G_2=\emptyset$,  contradicting with that $\mathcal{G}$ is intersecting. 

As mentioned early,  the shifting operation maintains the properties  
$|F\cap[\ell+s]|\geq s+1$, $|G\cap[\ell+s]|\geq s+1$ and $\binom{[\ell+s]}{\ell}\subseteq\mathcal{G}$. Combining with Lemma \ref{S3}, we may assume that $\mathcal{F}$ and $\mathcal{G}$ are shifted. 
Next, we are ready to apply induction on $n$ and $k$ and $\ell$.

For the case $n=k+\ell$,  the upper bound in Theorem \ref{main6} is $
\binom{k+\ell}{k}$, which simply follows from the cross-intersecting property of $\mathcal{F}$ and $\mathcal{G}$.

For the case $\ell = s+1$, we have $\mathcal{G} \subseteq\binom{[2s+1]}{s+1}$. It follows from $\binom{[2s+1]}{s+1}\subseteq\mathcal{G}$ that $\mathcal{G} =\binom{[2s+1]}{s+1}$. This, together with $ \mathcal{F} \subseteq\left\{A\in \binom{[n]}{k}: |A\cap[2s+1]|\geq s+1\right\}$, implies the result.

For the case $k=\ell$, it follows from the result in Conjecture \ref{co1} by setting $t=0$.

From now on, let us assume that $k >\ell> s+1$ and $n > k+\ell$. Note that
$\mathcal{F}(\bar{n})\subseteq\left\{A\in \binom{[n-1]}{k}: |A\cap[\ell+s]|\geq s+1\right\}$ and $\mathcal{G}(\bar{n})\subseteq\left\{B\in \binom{[n-1]}{\ell}: |B\cap[\ell+s]|\geq s+1\right\}$ are cross-intersecting.  In addition, $\mathcal{G}(\bar{n})$ is intersecting and $\binom{[\ell+s]}{\ell}\subseteq\mathcal{G}(\bar{n})$.
By induction hypothesis, we get
\begin{align}\label{f45}
|\mathcal{F}(\bar{n})|+|\mathcal{G}(\bar{n})| \leq\binom{\ell+s}{\ell}+\binom{n-1}{k}-\sum_{i=0}^s\binom{\ell+s}{i}\binom{n-\ell-s-1}{k-i}.
\end{align}
Since $n>\ell+s$ and $k >\ell> s+1$, we have 
$
\mathcal{F}(n)\subseteq\left\{A^*\in \binom{[n-1]}{k-1}: |A^*\cap[\ell+s]|\geq s+1\right\}
$
and
$
\mathcal{G}(n)\subseteq\left\{B^*\in \binom{[n-1]}{\ell-1}: |B^*\cap[\ell+s]|\geq s+1\right\}.
$
By Lemma \ref{S4}, $\mathcal{F}(n)$ and $\mathcal{G}(n)$ are cross-intersecting. Since $\ell+s>\ell-1+s$, by Theorem \ref{main4}, we get
\begin{align}\label{f46}
|\mathcal{F}(n)|+|\mathcal{G}(n)| \leq\binom{n-1}{k-1}-\sum_{i=0}^s\binom{\ell+s}{i}\binom{n-\ell-s-1}{k-i-1}.
\end{align}
The inequalities (\ref{f45}) and (\ref{f46}) together imply Theorem \ref{main6}.
$\hfill \square$\vspace{3mm}

\section*{Declaration of competing interest}
We declare that we have no conflict of interest to this work.

\section*{Data availability}
No data was used for the research described in the article.

\section*{Acknowledgement}
Lihua Feng was supported by 
the NSFC grant (Nos. 12271527 and 12471022). 
Yongtao Li was supported by the Postdoctoral Fellowship Program of CPSF (No. GZC20233196). 
This work was also partially supported by the NSF of Hunan Province (2023JJ30180) and the NSFC grant (No. 12201202). 
This paper is equally contributed. 
The authors would like to express their sincere thanks to the referees for the valuable suggestions, which greatly improved the presentation of the manuscript.

\end{document}